\documentclass[a4paper,11pt]{amsart}
\usepackage{extarrows}
\usepackage{amssymb}
\usepackage{mathrsfs}
\usepackage{amsmath,amssymb,amsthm,latexsym,amscd,mathrsfs}
\usepackage{indentfirst}
\usepackage{extarrows}
\usepackage[all]{xy}

\usepackage{multirow}

\numberwithin{equation}{section} \theoremstyle{plain}
\newtheorem{thm}{Theorem}[section]

\newtheorem{prop}[thm]{Proposition}

\newtheorem{cor}[thm]{Corollary}

\newtheorem{exmp}[thm]{Example}
\newtheorem{prob}[thm]{Problem}

\newtheorem{rem}[thm]{Remark}

\newtheorem{ack}{Acknowledgements}   
 \makeatletter

\newcommand{\fol}{\mathcal{F}}
\newcommand{\sphere}{\mathbb{S}}
\DeclareMathOperator{\Diff}{Diff}
\DeclareMathOperator{\Isom}{Isom}

\makeatother
\title[Foliations, diffeomorphisms and exotic structures]{Isoparametric foliations, diffeomorphism groups and exotic smooth structures}
\author[J.Q. Ge]{Jianquan Ge}
\address{School of Mathematical Sciences, Laboratory of Mathematics and Complex Systems, Beijing Normal
University, Beijing 100875, P.R. CHINA.}

\email{jqge@bnu.edu.cn}
\date{}
\thanks{The project was partially supported by the NSFC (No. 11331002), the Fundamental Research
Funds for the Central Universities, and a research fellowship from the Alexander von Humboldt Foundation. }
\begin{document}
\maketitle

\begin{abstract}
 In this paper, we are concerned with interactions between isoparametric theory and differential topology. Two foliations are called equivalent if there exists a diffeomorphism between the foliated manifolds mapping leaves to leaves. Using differential topology, we obtain several results towards the classification problem of isoparametric foliations up to equivalence. In particular, we show that each homotopy $n$-sphere has the ``same" isoparametric foliations as the standard sphere $\sphere^n$ has except for $n=4$, reducing the classification problem on homotopy spheres to that on the standard sphere. Moreover, we prove the uniqueness up to equivalence of isoparametric foliations with two points as the focal submanifolds on each sphere $\sphere^n$ except for $n=5$. Besides, we show that the uniqueness holds on $\sphere^5$ if and only if $\pi_0(\Diff(\sphere^4))\simeq\mathbb{Z}_2$, i.e., pseudo-isotopy implies isotopy for diffeomorphisms on $\sphere^4$. At last, some ideas behind the proofs enable us to discover new exotic smooth structures on certain manifolds.
\end{abstract}

\section{Introduction}\label{introduction}
A \emph{transnormal system} $\fol$ on a complete Riemannian manifold $M$ is a decomposition of $M$ into complete, injectively
immersed connected submanifolds, called \emph{leaves}, such that every geodesic emanating perpendicularly to one leaf remains perpendicular to all
leaves. A \emph{singular Riemannian foliation} is a transnormal system $\mathcal{F}$ which is also a \emph{singular foliation}, i.e., such that there are smooth vector fields $X_i$ on $M$ that span the tangent space $T_p L_p$ to the leaf $L_p$ through each point $p\in M$. A leaf of maximal dimension is called a \emph{regular} leaf, and its codimension is defined to be the \emph{codimension} of $\mathcal{F}$. Leaves of lower dimensions are called \emph{singular} leaves. By a \emph{foliated diffeomorphism} between two foliated manifolds we mean a diffeomorphism maps leaves to leaves, and such foliations are called \emph{equivalent}. By $M\cong M'$, $(M,\fol)\cong(M',\fol')$, we mean the manifolds are diffeomorphic, foliated diffeomorphic, respectively.

A singular Riemannian foliation $(M,\fol)$ of codimension 1 is called an \emph{isoparametric} foliation if the regular leaves have constant mean curvature. The regular leaves of an isoparametric foliation are called \emph{isoparametric hypersurfaces}, and the singular leaves are called \emph{focal submanifolds}.
The study of isoparametric foliations on the unit sphere $\sphere^n(1)$ originated in 1930's by E. Cartan and it developed into a very beautiful and valuable theory during the past decades. So far the classification in this case has been almost completed by a lot of contributions (for recent progress and applications see for example in \cite{CCJ07,Ch13,Mi13,TXY13,TY13} and a survey in \cite{Th10}).

Recall that a \emph{homotopy $n$-sphere} $\Sigma^n$ is a closed smooth manifold which has the homotopy type of $\sphere^n$. It is well-known (cf. \cite{JW08}) that $\Sigma^n$ is always homeomorphic to $\sphere^n$, but not always diffeomorphic in general, in which case $\Sigma^n$ is called an \emph{exotic sphere}. For $n\leq6$ and $n\neq4$, there are no exotic $n$-spheres. However, there are finitely many exotic $n$-spheres in infinitely many dimensions $n\geq7$. Recently, a remarkable result by Lytchak and Wilking \cite{LW13} shows that regular (i.e., all leaves are regular) Riemannian foliations are rather rare on homotopy spheres and in fact they occur only when the dimension of the
leaves is $1$, $3$ or $7$. In contrast, there are many singular Riemannian foliations of general codimension on unit spheres (cf. \cite{Ra14}).

As for codimension one case, Tang and the author \cite{GT12} started to study isoparametric foliations on exotic spheres and showed there are no isoparametric foliations on any exotic 4-spheres (if exist). There were also some existence examples presented on Milnor exotic 7-spheres and an example of codimension 1 singular Riemannian foliation with two points as the singular leaves on the Gromoll-Meyer 7-sphere. According to these we proposed the Problem 4.4 asking whether exotic $n$-spheres ($n\neq4$) always admit isoparametric foliations with the same focal submanifolds as those occurring on $\sphere^n$.

In \cite{QT13} Qian and Tang gave a fundamental construction of isoparametric foliations on closed manifolds that admit a decomposition into two linear disk\footnote{Throughout this paper, disks are assumed Euclidean closed disks and denoted by $D^n$ the $n$-dimensional disk; disk bundles are assumed linear, that is, closed disk bundles of vector bundles.} bundles of rank greater than 1 over closed submanifolds. On such manifolds they first constructed a Riemannian metric so that the canonical codimension 1 singular foliation (regular leaves correspond to concentric tubes around the zero sections) becomes a singular Riemannian foliation. Then they constructed a new (bundle-like) Riemannian metric so that the foliation becomes isoparametric. In particular, more examples of isoparametric foliations on exotic spheres analogues to those on standard spheres were obtained. We remark that, conversely, a codimension 1 singular Riemannian foliation on a closed simply connected manifold gives such a decomposition on the manifold (cf. \cite{Mol88}). Therefore, on closed simply connected manifolds isoparametric foliations require no more on the topology than codimension 1 singular Riemannian foliations.

Based on the Qian-Tang's fundamental construction, we can focus on the study of isoparametric foliations in the category of differential topology, neglecting the explicit Riemannian (bundle-like) metrics and geometries. In particular, classifications of isoparametric foliations on general manifolds (e.g., homotopy spheres) can be dealt with respect to foliated diffeomorphisms (equivalence classes) instead of isometries (congruence classes) as usually on unit spheres.

In this paper, we answer affirmatively the Problem 4.4 in \cite{GT12} mentioned above. In fact, we obtain
\begin{thm}\label{main thm}
 Each homotopy $n$-sphere ($n\neq4$) has the ``same" isoparametric foliations as the standard sphere $\sphere^n$ has, i.e., there is a one-to-one correspondence between the set of equivalence classes of isoparametric foliations on any homotopy $n$-sphere and that on the standard sphere $\sphere^n$. Moreover, each foliation decomposes the homotopy $n$-sphere into the same two disk bundles as those decomposed by the corresponding foliation on the standard sphere $\sphere^n$. In particular, They have the same (diffeomorphic) isoparametric hypersurfaces and focal submanifolds.
\end{thm}

As a consequence, exotic $n$-spheres ($n\neq4$) admit the same disk bundle decompositions as the standard sphere $\sphere^n$ does, in analogy with the well-known result of S. Smale that each exotic $n$-sphere ($n\neq4$) is a twisted sphere (a homotopy sphere obtained by gluing two disks along the boundaries).

By Theorem \ref{main thm}, to study isoparametric foliations on homotopy $n$-spheres it suffices to study on $\sphere^n$.
However, the classification of equivalence classes of isoparametric foliations on $\mathbb{S}^n$ ($n>4$) is far from completed. Notice that the known classification results on $\sphere^n$ are restricted to the round metric and congruence classes. Different metrics can admit different isoparametric foliations (see Examples \ref{examples}), while different congruence classes might belong to the same equivalence class.

Section \ref{disk bundle section} is mainly devoted to discussing the classification problem of isoparametric foliations up to equivalence on a fixed manifold. Via introducing some subgroups of diffeomorphism groups, we first develop some criterions to distinguish two foliations up to equivalence in Subsection \ref{subsect-criterion}. According to these criterions, in Subsection \ref{subsect-classification} we propose three main steps towards the classification. About the steps $(2)$ and $(3)$, we obtain some necessary conditions for a manifold admitting non-equivalent foliations with the same disk bundle decomposition, and also give some sufficient conditions to construct such examples in Theorem \ref{negative ex} and Corollary \ref{E+=Dn}, which can be applied quite naturally to the Eells-Kuiper quaternionic projective planes and their $SC^p$ Riemannian structures (cf. Remark \ref{Eells-Kuiper}). As an application, in Theorem \ref{cor-twopoints isop} we obtain the following rigidity result for foliations on spheres, answering partially the classification problem on $\sphere^n$:
\begin{itemize}
\item[$\bullet$] Every sphere $\sphere^n$ $(n\neq5)$ admits exactly one equivalence class of isoparametric foliations with two points as the focal submanifolds.
\item[$\bullet$] $\pi_0(\Diff(\sphere^{4}))\neq\mathbb{Z}_2$ if and only if $\sphere^5$ admits non-equivalent isoparametric foliations with two points as the focal submanifolds.
    \end{itemize}

Recall that two diffeomorphisms (resp. embeddings) $f_i: M\rightarrow N$ $(i=0,1)$ are called \emph{pseudo-isotopic} (or \emph{concordant}, \emph{quasi-isotopic}, \emph{quasi-diffeotopic}), if they extend to a diffeomorphism (resp. embedding) $F: M\times [0,1]\rightarrow N\times [0,1]$, and are called \emph{isotopic} (or \emph{diffeotopic}) if in addition $F$ is level-preserving (cf. \cite{Wa63}). It worths mentioning that through several remarkable contributions by Kervaire and Milnor \cite{KM63}, Cerf \cite{Ce70}, Smale \cite{Sm59} and Hatcher \cite{Ha83}, etc., $\pi_0(\Diff^+(\sphere^{n}))$ $(n\neq4)$ have been understood well (calculated explicitly in low dimensions) and in particular, any two pseudo-isotopic diffeomorphisms of $\sphere^{n}$ $(n\neq4)$ must be isotopic to each other. Nevertheless, it is still unknown whether $\pi_0(\Diff(\sphere^{4}))=\mathbb{Z}_2$, or equivalently, whether pseudo-isotopy implies isotopy for diffeomorphisms on $\sphere^4$ as in other dimensions (see Remarks \ref{rem-diffSn}, \ref{rem-pseudoisotop}). Hence, by the second assertion of Theorem \ref{cor-twopoints isop} above, this question is equivalent to asking whether $\sphere^5$ also admits unique equivalence class of isoparametric foliations with two points as the focal submanifolds as in other dimensions.

At the last of Section \ref{disk bundle section}, we present new examples of isoparametric foliations on $\sphere^n$ that are non-equivalent to the classical isoparametric foliations on the unit sphere $\sphere^n(1)$ (see Example \ref{examples}), illustrating the step (1) in the classification problem.

In Section \ref{sect-isop-sphere} we prove Theorem \ref{main thm} by using the h-cobordism theorem, a cutting-gluing surgery and the criterions developed in Section \ref{disk bundle section}.
Inspired by this proof, we apply the cutting-gluing surgery to the study of inertia groups and exotic smooth structures in Section \ref{sect-inertia}. Firstly we observe a new relation (\ref{I1<I0}), in Theorem \ref{thm-I1<I0}, between the two inertia groups $I_1(M)$ and $I_0(N)$ when $M$ is a hypersurface of $N$. Combining this with the original rigid relation (\ref{I1=I0}) proven by Levine \cite{Le70}, we see that $M\times\sphere^1$ has the smallest inertia group $I_0$ among all manifolds containing $M$ as a hypersurface. As an application, we obtain
\begin{thm}\label{exoticstr-thm}
Let $M$ be a closed hypersurface embedded in $\sphere^n$. Then for any $k\geq0$ and any product $P^k:=\sphere^{k_1}\times\cdots\times\sphere^{k_l}$ of standard spheres of total dimension $k=\sum_{i=1}^l k_i$ ($k_i\geq1$) ($P^k$ is a point when $k=0$), there exist at least $|\Theta_{n+k}|$ distinct oriented smooth structures on $M^{n-1}\times P^k\times\sphere^1$, where $|\Theta_n|$ is the order of the finite abelian group $\Theta_n$ of h-cobordism classes of oriented homotopy $n$-spheres.
\end{thm}

It also follows that the group $\Gamma(M^{n-1}\times P^k)$ of concordance (pseudo-isotopy) classes of orientation-preserving diffeomorphisms of $M^{n-1}\times P^k$ (and hence the mapping class group $\pi_0(\Diff^+(M^{n-1}\times P^k))$) has at least $|\Theta_{n+k}|$ elements.

\section{Disk bundle decomposition by singular Riemannian foliation}\label{disk bundle section}
A codimension 1 singular Riemannian foliation $\fol$ on a closed simply connected manifold $N$ has exactly two closed singular leaves $M_{\pm}$ and decomposes $N$ into two unit\footnote{In general, the radius can be a positive constant which can be normalized by a homothetic transformation of the Riemannian metric on $N$.} disk bundles $E_{\pm}$ of the normal vector bundles $\xi_{\pm}$ over $M_{\pm}$ of rank $m_{\pm}>1$ (cf. \cite{Mol88}). The decomposition can be described by the following commutative diagram
\begin{equation}\label{gluing str}
\xymatrix{
 &E:=E_{+}\sqcup E_{-}\ar[dl]_{p}\ar[dr]^{\pi}&\\
 E_{\varphi}:=E_{+}\cup_{\varphi}E_{-}\ar[rr]_-{\widetilde{\pi}}^-{\cong}&&N
 =N_{+}\cup N_{-}.
}
 \end{equation}
  Here $E$ is the disjoint union of $E_{\pm}$, $\pi_{\pm}:=\pi|_{E_{\pm}}: E_{\pm}\xlongrightarrow{\cong}N_{\pm}$ are closed tubular neighborhoods of the singular leaves $\iota_{\pm}:M_{\pm}\rightarrow N$ with $s_{\pm}:=\pi_{\pm}^{-1}\circ\iota_{\pm}:M_{\pm}\rightarrow E_{\pm}$ the zero sections of $E_{\pm}$, $\varphi=\pi_{-}^{-1}\circ\pi_{+}|_{\partial E_{+}}:\partial E_{+}\rightarrow\partial E_{-}$ is the gluing diffeomorphism for $E_{\varphi}$, $\widetilde{\pi}$ is the diffeomorphism from $E_{\varphi}$ to $N$ whose composition with the natural projection $p$ satisfies $\widetilde{\pi}\circ p=\pi$. The regular leaves of $\fol$ are the images of the concentric tubes around the zero sections of $E_{\pm}$ (of constant radii under the induced Euclidean metrics of $E_{\pm}$) under the maps $\pi_{\pm}$, and the singular leaves are the images of the zero sections. The preimage of $\fol$ under $\widetilde{\pi}$ defines a codimension 1 singular Riemannian foliation $\fol_{\varphi}$ on $E_{\varphi}$ with the induced metric by $\widetilde{\pi}$. The leaves of $\fol_{\varphi}$ are just the concentric tubes (including the tubes of radius $0$, the zero sections) in $E_{\pm}$. Therefore, the equivalence class of $(N,\fol)$ can be represented by $(E_{\varphi},\fol_{\varphi})$.

  Conversely, given unit disk bundles $E_{\pm}$ over complete manifolds $M_{\pm}$ and a diffeomorphism $\varphi:\partial E_{+}\rightarrow\partial E_{-}$, the foliation $\fol_{\varphi}$ consisting of concentric tubes on $E_{\varphi}=E_{+}\cup_{\varphi}E_{-}$ would be a singular Riemannian foliation, provided with a Riemannian metric by a suitable choice of a one-parameter family of metrics on $\partial E_{+}$ in a collar of $\partial E_{+}$ in $E_+$, connecting $g_{+}|_{\partial E_{+}}$ and $\varphi^{*}(g_{-}|_{\partial E_{-}})$, where $g_{\pm}$ are metrics on $E_{\pm}$ compatible with the Euclidean metrics (cf. \cite{QT13}). Moreover, if $E_{\pm}$ are of rank greater than 1 and $M_{\pm}$ are closed, $(E_{\varphi},\fol_{\varphi})$ can become isoparametric by a more careful choice of the one-parameter family of metrics on $\partial E_{+}$ as shown in \cite{QT13}.

  It follows that to study classification of equivalence classes of codimension 1 singular Riemannian (isoparametric) foliations one needs only to study the foliations in the form $(E_{\varphi},\fol_{\varphi})$ determined by pairs of unit disk bundles $E_{\pm}\subset\xi_{\pm}$ with diffeomorphic boundaries and gluing diffeomorphisms $\varphi:\partial E_+\rightarrow\partial E_-$. Moreover, it is independent of the choices of the vector bundles $\xi_{\pm}$ in their bundle-equivalence (i.e., isomorphism) classes and of the Euclidean metrics. In fact, for any Euclidean bundles $\xi_{\pm}'$ isomorphic to $\xi_{\pm}$, there are vector bundle isomorphisms $F_{\pm}:\xi_{\pm}\rightarrow \xi_{\pm}'$ which are isometries with respect to the Euclidean metrics (cf. \cite{MS74}) and hence map concentric tubes of $E_{\pm}$ to concentric tubes of $E_{\pm}'$. Set $\psi=F_-\circ\varphi\circ F_+^{-1}|_{\partial E_+'}:\partial E_+'\rightarrow \partial E_-'$. Then the map \[F:E_{\varphi}=E_+\cup_\varphi E_-\rightarrow E_+'\cup_{\psi}E_-'=E_{\psi}'\]
  defined by $F|_{E_{\pm}}=F_{\pm}$ is a foliated diffeomorphism between $(E_{\varphi},\fol_{\varphi})$ and $(E_{\psi}',\fol_{\psi}')$.
  \subsection{Equivalence criterions for codimension 1 singular Riemannian foliations.}\label{subsect-criterion}
  Motivated by the discussion above, we observe the following criterion to distinguish two foliations up to equivalence.
 \begin{prop}\label{equivalence criterion}
For $\varphi_0:\partial E_{+}\rightarrow\partial E_{-}$, $\varphi_1:\partial \widetilde{E}_{+}\rightarrow\partial \widetilde{E}_{-}$,  $(E_{\varphi_0},\fol_{\varphi_0})$ is foliated diffeomorphic to $(\widetilde{E}_{\varphi_1},\fol_{\varphi_1})$ if and only if there are diffeomorphisms $F_{\pm}:E_{\pm}\rightarrow \widetilde{E}_{\pm}$ mapping concentric tubes to concentric tubes such that $\varphi_1=F_{-}\circ\varphi_0\circ F_{+}^{-1}|_{\partial \widetilde{E}_+}$.
 \end{prop}
\begin{proof}
 The conclusion follows directly from the definitions and can be described by the commutative diagram
 \begin{equation*}
\xymatrix{
 E=E_+\sqcup E_-\ar[d]_{p_0}\ar[r]^{F_+\sqcup F_-}&\widetilde{E}=\widetilde{E}_+\sqcup\widetilde{E}_-\ar[d]^{p_1}\\
 (E_{\varphi_0},\fol_{\varphi_0})\ar[r]_-{F}^-{\cong}&(\widetilde{E}_{\varphi_1},\fol_{\varphi_1}),
}
 \end{equation*}
 where $p_0$ (resp. $p_1$) is the natural projection mapping $x\in\partial E_+$ (resp. $\partial \widetilde{E}_+$) and $\varphi_0(x)\in\partial E_-$ (resp. $\varphi_1(x)\in\partial \widetilde{E}_-$) to the gluing point in $E_{\varphi_0}$ (resp. $\widetilde{E}_{\varphi_1}$), $F$ is the foliated diffeomorphism satisfying $F\circ p_0|_{E_{\pm}}=p_1\circ F_{\pm}$.
\end{proof}
If $\varphi_i:\partial E_+\rightarrow\partial E_-$ $(i=0,1)$ are isotopic diffeomorphisms, by considering a collar $C$ of $\partial E_-$ in $E_-$ (resp. $\partial E_+$ in $E_+$) we can extend $\varphi_1\circ\varphi_0^{-1}$ (resp. $\varphi_1^{-1}\circ\varphi_0$) to a diffeomorphism $F_-\in\Diff(E_{-})$ (resp. $F_+\in\Diff(E_{+})$) preserving concentric tubes, and then $\varphi_1=F_-\circ \varphi_0$ (resp. $\varphi_1=\varphi_0\circ F_+^{-1}$), hence $(E_{\varphi_0},\fol_{\varphi_0})\cong (E_{\varphi_1},\fol_{\varphi_1})$. This shows
\begin{cor}\label{isotopy invariant}
The equivalence class of $(E_{\varphi},\fol_{\varphi})$ is independent of the choice of $\varphi$ in its isotopy class.
\end{cor}

Motivated by the criterion in Proposition \ref{equivalence criterion}, we introduce some subgroups of the diffeomorphism groups as follows.
Let $\pi:E_1\rightarrow B$ be the unit disk bundle of a Euclidean vector bundle $\xi$ over a complete connected manifold $B$. Let $\Diff_c(E_1)$ denote the subgroup $\Diff(E_1,\fol_c)$ of $\Diff(E_1)$ consisting of foliated diffeomorphisms from $(E_1,\fol_c)$ to itself, where $\fol_c$ is the canonical foliation consisting of concentric tubes $T_t$ of constant radii $t\in[0,1]$ around the zero section. Let $\Isom_b(E_1)$ be the subgroup of $\Diff_c(E_1)$ consisting of smooth (self-)bundle maps\footnote{A smooth bundle map between two vector bundles carries each vector space isomorphically onto a vector space, inducing a diffeomorphism of the base manifolds. When the induced diffeomorphism on the base is the identity map, the bundle map is called an isomorphism. The induced bundle by the bundle map is isomorphic to the domain bundle. Therefore, bundle maps can also be  regarded as an equivalence relation for vector bundles as isomorphisms do. See \cite{MS74}.} preserving the Euclidean metric. Denote by $\rho:\Diff(E_1)\rightarrow \Diff(\partial E_1)$ the homomorphism maps $F\in \Diff(E_1)$ to $\rho(F)=F|_{\partial E_1}$. Then $\Diff_c(\partial E_1):=\rho(\Diff_c(E_1))$, $\Isom_b(\partial E_1):=\rho(\Isom_b(E_1))$ are two subgroups in $\Diff_{E_1}(\partial E_1):=\rho(\Diff(E_1))$, the subgroup of $\Diff(\partial E_1)$ consisting of diffeomorphisms extendable to $E_1$.
\begin{prop}\label{linearization}
The inclusions $\Isom_b(E_1)\hookrightarrow\Diff_c(E_1)$, $\Isom_b(\partial E_1)\hookrightarrow\Diff_c(\partial E_1)$ are bijections on path components, i.e., $\pi_0(\Isom_b(E_1))\simeq \pi_0(\Diff_c(E_1))$, $\pi_0(\Isom_b(\partial E_1))\simeq\pi_0(\Diff_c(\partial E_1))$.
In particular, any diffeomorphism in $\Diff_c(\partial E_1)$ (resp. $\Diff_c(E_1)$) is isotopic to one in $\Isom_b(\partial E_1)$ (resp. $\Isom_b(E_1)$).
\end{prop}
\begin{proof}
For $t\in(0,1]$, let $h_t: E_1\rightarrow E_1$ be the dilatation mapping $V\in E_1$ to $tV$, and for any $F\in\Diff_c(E_1)=\Diff(E_1,\fol_c)$, we define \[F_t:=(F\circ h_t)/\lambda(t)\in\Diff_c(E_1),\]
 where $\lambda:[0,1]\rightarrow[0,1]$ is the function defined by $F(T_t)=T_{\lambda(t)}$ for any concentric tube $T_t$ of radius $t\in[0,1]$. Explicitly, we have \[\lambda(t)=|F(tV)|, \quad for~any~V\in\partial E_1,\] where $|\cdot|$ denotes the norm of the Euclidean metric, and for $t=0$, $tV=\pi(V)$ means the base point in the zero section $B$. In fact,
 it is not hard to verify that \[\lambda(t)=\int_0^t|(\gamma_V'(t))^{\perp}|dt, \quad for~any~V\in\partial E_1,\] where $\gamma_V'(t)$ is the tangent vector of the curve $\gamma_V(t):=F(tV)$, and $(\cdot)^{\perp}$ means the projection from $T(E_1)$ to the vertical distribution $E_1$. Moreover, $\gamma_V'(t)=(F_*)_{tV}(V)$ is the image of $V$ under the tangential map $F_*:T(E_1)\rightarrow T(E_1)$ at $tV\in E_1$. Therefore, $\lambda$ is a smooth function with \[\lambda'(t)=|(\gamma_V'(t))^{\perp}|=|((F_*)_{tV}(V))^{\perp}|\geqq0,\quad \lambda'(0)>0,\quad for~any~V\in\partial E_1.\]
 Define $F_0:E_1\rightarrow E_1$ by \[F_0(V):=((F_*)_{\pi(V)}(V))^{\perp}/\lambda'(0),\quad for~V\in E_1. \] Then since $F_*$ is linear, $F_0$ is linear and preserves lengths, thus $F_0\in \Isom_b(E_1)$.

 Set $f_t:=\rho(F_t)=F_t|_{\partial E_1}\in \Diff_c(\partial E_1)$, and $f_0:=\lim_{t\rightarrow{0+}}f_t$. It follows that \[f_0(V)=\lim_{t\rightarrow{0+}}F(tV)/\lambda(t)=((F_*)_{\pi(V)}(V))^{\perp}/\lambda'(0)=F_0(V),\quad for~V\in\partial E_1.\]
 Hence $f_0=\rho(F_0)\in \Isom_b(\partial E_1)$.

In conclusion, we have shown that any $F\in\Diff_c(E_1)$ corresponds uniquely to a path $\{f_t|t\in[0,1]\}$ in $\Diff_c(\partial E_1)$ with one end $f_0\in\Isom_b(\partial E_1)$, which also shows that $\Isom_b(\partial E_1)\hookrightarrow\Diff_c(\partial E_1)$ is a bijection on path components. Conversely, given such a path $\{f_t|t\in[0,1]\}$ and a smooth nondecreasing function $\lambda:[0,1]\rightarrow[0,1]$ with $\lambda'(0)>0$, one can construct an $F\in\Diff_c(E_1)$ by: for $V\in\partial E_1$, $F(tV)=\lambda(t)f_t(V)$ for $t\in(0,1]$, and $F(b)=F_0(b)$ for $b=\pi(V)\in B$, where $F_0$ is the isomorphism that restricts to $f_0$. Therefore, a retraction of the path $\{f_t|t\in[0,1]\}$ to the constant path $\{\tilde{f}_t\equiv f_0\}$ induces a path connecting $F\in\Diff_c(E_1)$ and $F_0\in\Isom_b(E_1)$ in $\Diff_c(E_1)$. The proof is now complete.
\end{proof}

\begin{cor}\label{Diffe-Diffc}
Either the inclusion $\Isom_b(\partial E_1)\hookrightarrow \Diff_{E_1}(\partial E_1)$ is not surjective on path components, i.e., $\pi_0(\Isom_b(\partial E_1))\varsubsetneqq \pi_0(\Diff_{E_1}(\partial E_1))$, or $\Diff_{E_1}(\partial E_1)=\Diff_c(\partial E_1)$.
\end{cor}
\begin{proof}
If $\pi_0(\Isom_b(\partial E_1))\simeq\pi_0(\Diff_{E_1}(\partial E_1))$, then $\pi_0(\Diff_{E_1}(\partial E_1))\simeq \pi_0(\Diff_c(\partial E_1))$ by Proposition \ref{linearization}. In particular, any $f\in\Diff_{E_1}(\partial E_1)$ is isotopic to one in $\Diff_c(\partial E_1)$. Then by considering a collar of $\partial E_1$ in $E_1$, the isotopy induces an extension of $f$ to a diffeomorphism $F\in\Diff_c(E_1)$, proving that $f=F|_{\partial E_1}\in\Diff_c(\partial E_1)$.
\end{proof}

Notice that Proposition \ref{linearization} also holds for maps between two disk bundles. Thus by Proposition \ref{equivalence criterion} we have
\begin{cor}\label{non-isomorphic bundles}
For $\varphi_0:\partial E_{+}\rightarrow\partial E_{-}$, $\varphi_1:\partial \widetilde{E}_{+}\rightarrow\partial \widetilde{E}_{-}$,  $(E_{\varphi_0},\fol_{\varphi_0})$ is foliated diffeomorphic to $(\widetilde{E}_{\varphi_1},\fol_{\varphi_1})$ if and only if there are smooth bundle maps $F_{\pm}:E_{\pm}\rightarrow \widetilde{E}_{\pm}$ (preserving the Euclidean metrics) such that $\varphi_1$ is isotopic to $F_{-}\circ\varphi_0\circ F_{+}^{-1}|_{\partial \widetilde{E}_+}$.
\end{cor}
\begin{proof}
By Propositions \ref{equivalence criterion} and \ref{linearization}, we can get the conclusion directly provided with $F_{\pm}$ preserving the Euclidean metrics. Nevertheless, the latter condition is superfluous since, otherwise, one could composite $F_{\pm}$ with some (self-)isomorphisms of $E_{\pm}$ so as to make them preserving the Euclidean metrics, which results in equivalent foliations (see the discussion before this subsection).
\end{proof}
As an immediate application, this will lead to many unexpected examples of non-equivalent foliations on $\sphere^n$. We postpone the discussion in the next subsection.

\subsection{Towards a classification of foliations on a fixed manifold.}\label{subsect-classification}
In order to classify codimension 1 singular Riemannian (isoparametric) foliations up to equivalence on a closed simply connected manifold $N$ (e.g., $\sphere^n$), the preceding discussions, in particular Corollary \ref{non-isomorphic bundles}, suggest the following three steps:
\begin{itemize}
\item[(1)] Classify the disk bundle pairs $E_{\pm}$ (up to bundle maps or isomorphisms) such that $N\cong E_+\cup_{\varphi}E_-=:E_{\varphi}$ for some gluing diffeomorphism $\varphi:\partial E_+\rightarrow\partial E_-$.
\item[(2)] Classify the isotopy classes of the gluing diffeomorphisms $\varphi$ such that $N\cong E_{\varphi}$.
\item[(3)] Take the quotient of the set $G_N$ of isotopy classes in step (2) by the action $\beta:\pi_0(\Isom_b(E_{\pm}))\times G_N\rightarrow G_N$, $([F_{\pm}], [\varphi])\mapsto [F_{-}\circ\varphi\circ F_{+}^{-1}|_{\partial E_+}]$. Then for a given pair $E_{\pm}$ in step (1), we have exactly $|G_N/\beta|$ non-equivalent foliations.\footnote{Here we have not assumed orientations on the manifolds and foliations, hence foliated diffeomorphisms need not to be orientation-preserving. If necessary, one can also incorporate orientation into the classification problem (see Corollary \ref{E+=Dn}).}
\end{itemize}

At first glance, the step (2) seems redundant as one would imagine that only one isotopy class of $\varphi$ could produce $N$. In fact, non-isotopic gluing diffeomorphisms may also give rise to diffeomorphic manifolds. For example, if there exist non-isotopic but pseudo-isotopic diffeomorphisms $\varphi_i:\partial E_+\rightarrow\partial E_-$, then they generate diffeomorphic manifolds. In the following we discuss the steps (2) and (3). Namely, we are now concerned with how many (or how to distinguish) foliations up to equivalence would be derived from a (fixed) disk bundle pair $E_{\pm}$.

Firstly, if two gluing diffeomophisms give non-diffeomorphic manifolds $E_{\varphi_0}, E_{\varphi_1}$, then clearly the foliations $(E_{\varphi_0},\fol_{\varphi_0}), (E_{\varphi_1},\fol_{\varphi_1})$ will be non-equivalent. But once the two glued manifolds are diffeomorphic (as in the steps (2) and (3)), one cannot distinguish the foliations in general though his intuition tells so. In fact, for the easiest case that $E_{\pm}=D^n$, we will prove Theorem \ref{cor-twopoints isop} below which shows the complexity of this question. Precisely, we are now concerned with the following
\begin{prob}\label{problem}
For $\varphi_i:\partial E_{+}\rightarrow\partial E_{-}$ $(i=0,1)$ satisfying $E_{\varphi_0}\cong E_{\varphi_1}$, when is $(E_{\varphi_0},\fol_{\varphi_0})$ foliated diffeomorphic to $(E_{\varphi_1},\fol_{\varphi_1})$ (or how to distinguish them)?
\end{prob}
 For instance, it holds when $E_{\varphi_i}$ are closed simply connected $4$-manifolds as shown in the classification by the author and Radeschi \cite{GR13}.

Let $\varphi_0:\partial E_+\rightarrow\partial E_-$ be a gluing diffeomorphism and $(E_{\varphi_0}, \fol_{\varphi_0})$ be the foliation as before.
For any $h_{\pm}\in \Diff_{E_{\pm}}(\partial E_{\pm})$, it is easily seen that the glued manifold $E_{h_-\circ\varphi_0\circ h_+^{-1}}$ is diffeomorphic to $E_{\varphi_0}$. The foliations $(E_{h_-\circ\varphi_0\circ h_+^{-1}},\fol_{h_-\circ\varphi_0\circ h_+^{-1}})$ are then candidates for foliations non-equivalent to $(E_{\varphi_0},\fol_{\varphi_0})$ on diffeomorphic manifolds, giving negative examples to Problem \ref{problem}. In the following we show that the groups $\pi_0(\Isom_b(\partial E_{\pm}))$, $\pi_0(\Diff_{E_{\pm}}(\partial E_{\pm}))$ play key role in voting for them.
\begin{thm}\label{negative ex}
 With notations as before, we have

$(1)$ If there were $h_{\pm}\in \Diff_{E_{\pm}}(\partial E_{\pm})$ such that $(E_{h_-\circ\varphi_0\circ h_+^{-1}},\fol_{h_-\circ\varphi_0\circ h_+^{-1}})\ncong(E_{\varphi_0}, \fol_{\varphi_0})$, then either $\pi_0(\Isom_b(\partial E_+))\varsubsetneqq \pi_0(\Diff_{E_{+}}(\partial E_+))$ or $\pi_0(\Isom_b(\partial E_-))\varsubsetneqq \pi_0(\Diff_{E_{-}}(\partial E_-))$.

$(2)$ Consider $E_+=E_-$. If $\pi_0(\Isom_b(\partial E_+))\varsubsetneqq \pi_0(\Diff_{E_{+}}(\partial E_+))$, then for any $[\varphi_0]\in\pi_0(\Isom_b(\partial E_+))$ and $[\varphi_1]\in\pi_0(\Diff_{E_{+}}(\partial E_+))\setminus\pi_0(\Isom_b(\partial E_+))$, we have $E_{\varphi_0}\cong E_{\varphi_1}$ but $(E_{\varphi_0},\fol_{\varphi_0})\ncong(E_{\varphi_1},\fol_{\varphi_1})$.
\end{thm}
\begin{proof}
(1) We prove this by contradiction. If both $\pi_0(\Isom_b(\partial E_{\pm}))\simeq\pi_0(\Diff_{E_{\pm}}(\partial E_{\pm}))$, then $\Diff_{E_{\pm}}(\partial E_{\pm})=\Diff_c(\partial E_{\pm})$ by Corollary \ref{Diffe-Diffc}, whence $(E_{h_-\circ\varphi_0\circ h_+^{-1}},\fol_{h_-\circ\varphi_0\circ h_+^{-1}})\cong(E_{\varphi_0}, \fol_{\varphi_0})$ for any $h_{\pm}\in \Diff_{E_{\pm}}(\partial E_{\pm})$ by Proposition \ref{equivalence criterion}.

(2) Under the assumptions, we see that $h:=\varphi_1\circ\varphi_0^{-1}$ lies in $\pi_0(\Diff_{E_{+}}(\partial E_+))\setminus\pi_0(\Isom_b(\partial E_+))$, which shows immediately $E_{\varphi_0}\cong E_{h\circ\varphi_0}=E_{\varphi_1}$. Now we prove the nonequivalence of the foliations by contradiction.

If $(E_{\varphi_0},\fol_{\varphi_0})\cong(E_{\varphi_1},\fol_{\varphi_1})$, then by Proposition \ref{equivalence criterion} there exist $f_{\pm}\in\Diff_c(\partial E_+)$ such that $\varphi_1=f_-\circ\varphi_0\circ f_+^{-1}$. Hence $h=f_-\circ\varphi_0\circ f_+^{-1}\circ\varphi_0^{-1}$. By Proposition \ref{linearization}, $[f_{\pm}]\in\pi_0(\Isom_b(\partial E_+))$ and hence we have $[h]=[f_-\circ\varphi_0\circ f_+^{-1}\circ\varphi_0^{-1}]\in\pi_0(\Isom_b(\partial E_+))$, the contradiction.
\end{proof}
\begin{rem}\label{rem-diffSn}
Consider $E_+=D^n$. $\Isom_b(\partial E_+)=\Isom(\sphere^{n-1})=O(n)$ and $\Diff_{E_+}(\partial E_+)=\rho(\Diff(D^n))=\Diff_{D^n}(\sphere^{n-1})\subseteq \Diff(\sphere^{n-1})$. For $n\geq 6$, it is well-known from the pseudo-isotopy theorem of Cerf (cf. \cite{Ce70}) that $\pi_0(\Diff^+(D^n))=0$, and thus $\pi_0(\Diff_{D^n}(\sphere^{n-1}))\simeq \pi_0(\Isom(\sphere^{n-1}))\simeq \mathbb{Z}_2$. For $n\leq5$, $\Diff_{D^n}(\sphere^{n-1})=\Diff(\sphere^{n-1})$ because of the exact sequence $\pi_0(\Diff^+(D^n))\xlongrightarrow{\rho}\pi_0(\Diff^+(\sphere^{n-1}))\rightarrow \Gamma_n$, where the group of twisted $n$-spheres $\Gamma_n=0$ for $n\leq6$. Therefore, for $n\leq4$, we have also $\pi_0(\Diff_{D^n}(\sphere^{n-1}))\simeq \pi_0(\Isom(\sphere^{n-1}))\simeq \mathbb{Z}_2$, since $\Diff(\sphere^{n-1})\approx O(n)$ are homotopy equivalent proven by Smale \cite{Sm59} for $n=3$ and by Hatcher \cite{Ha83} for $n=4$ (well-known for $n\leq2$). For the case of $n=5$,  it is an open problem whether or not $\Diff(\sphere^4)$ has more than two components.
\end{rem}
\begin{rem}\label{rem-pseudoisotop}
 It follows from the remark above that for $n\neq5$, the group $\Gamma(\sphere^{n-1})$ of concordance (pseudo-isotopy) classes of orientation-preserving diffeomorphisms of $\sphere^{n-1}$ is equal to $\pi_0(\Diff^+(\sphere^{n-1}))$, i.e., pseudo-isotopy implies isotopy for diffeomorphisms on $\sphere^{n-1}$; and for $n=5$, $\Gamma(\sphere^{4})\simeq0\subseteq \pi_0(\Diff^+(\sphere^{4}))$. In particular, it is not known whether there exist non-isotopic but pseudo-isotopic diffeomorphisms on $\sphere^{4}$.
\end{rem}

  If $E_+=D^n$, by the disk theorem of Palais \cite{Pa60} (two orientation-preserving embeddings of the disk are isotopic), every pair $E_{\varphi_0}\cong E_{\varphi_1}$ has the form in (1) of Theorem \ref{negative ex}, i.e., $\varphi_1=h_-\circ\varphi_0\circ h_+^{-1}$ for some $h_{\pm}\in \Diff_{E_{\pm}}(\partial E_{\pm})$. This partially addresses the step (2) in the classification problem, namely, the set $G_N$ ($N\cong E_{\varphi_0}$) of isotopy classes of diffeomorphisms $\varphi:\partial E_+=\sphere^{n-1}\rightarrow \partial E_-$ satisfying $N\cong E_{\varphi}$ is equal to $$\Big\{[\varphi] \mid \varphi=h_-\circ\varphi_0\circ h_+^{-1} \quad \textit{for any}~~ h_{\pm}\in \Diff_{E_{\pm}}(\partial E_{\pm}) \Big\}.$$
  Moreover, it is easy to show that each orientation-preserving diffeomorphism $h_{+}\in \Diff_{E_{+}}^+(\partial E_{+})=\Diff_{D^n}^+(\sphere^{n-1})$ is pseudo-isotopic to the identity and vice versa. If in addition $E_-=D^n$ and orientations assumed (now $N$ is an oriented homotopy sphere), then the set $G_N$ consists of those isotopy classes which are pseudo-isotopic to $\varphi_0$. Combining these discussions with Theorem \ref{negative ex} and Remarks \ref{rem-diffSn}, \ref{rem-pseudoisotop}, we are able to derive the following construction and rigidity results (\ref{E+=Dn}, \ref{cor-twopoints isop}), addressing partially the step (3) in the classification problem.
   \begin{cor}\label{E+=Dn}
   Suppose $E_+=D^n$.

   $(1)$ For $n\neq5$, for any $\varphi_i:\partial E_{+}\rightarrow\partial E_{-}$ $(i=0,1)$ satisfying $E_{\varphi_0}\cong E_{\varphi_1}$,  $(E_{\varphi_0},\fol_{\varphi_0})\ncong(E_{\varphi_1},\fol_{\varphi_1})$ holds only if $\pi_0(\Isom_b(\partial E_-))\varsubsetneqq \pi_0(\Diff_{E_{-}}(\partial E_-))$.

   $(2)$ If $\pi_0(\Isom_b(\partial E_-))\varsubsetneqq \pi_0(\Diff_{E_{-}}(\partial E_-))$, then $\pi_0(\Isom_b^+(\partial E_-))\varsubsetneqq \pi_0(\Diff_{E_{-}}^+(\partial E_-))$, and for any $\varphi_i:\partial E_{+}\rightarrow\partial E_{-}$ $(i=0,1)$ with $[\varphi_1\circ\varphi_0^{-1}]\in\pi_0(\Diff_{E_{-}}^+(\partial E_-))\setminus\pi_0(\Isom_b^+(\partial E_-))$, the oriented foliations $(E_{\varphi_i},\fol_{\varphi_i})$ $(i=0,1)$ are not oriented foliated diffeomorphic to each other though $E_{\varphi_0}\cong E_{\varphi_1}$ orientation-preserving.
   \end{cor}
   \begin{proof}
   The first assertion follows directly from (1) of Theorem \ref{negative ex} and $\pi_0(\Diff_{D^n}(\sphere^{n-1}))\simeq \pi_0(\Isom(\sphere^{n-1}))$ $(n\neq5)$ stated in Remark \ref{rem-diffSn}.

   For the assertion (2), firstly, it is clear $\pi_0(\Isom_b^+(\partial E_-))\varsubsetneqq \pi_0(\Diff_{E_{-}}^+(\partial E_-))$, and we prove the latter claim by contradiction. If $(E_{\varphi_0},\fol_{\varphi_0})\cong(E_{\varphi_1},\fol_{\varphi_1})$ orientation-preserving, then by Proposition \ref{equivalence criterion} there exist $f_{\pm}\in\Diff_c^+(\partial E_{\pm})$ such that $\varphi_1=f_-\circ\varphi_0\circ f_+^{-1}$. Hence $\varphi_1\circ\varphi_0^{-1}=f_-\circ\varphi_0\circ f_+^{-1}\circ\varphi_0^{-1}$. By Proposition \ref{linearization}, $[f_{\pm}]\in\pi_0(\Isom_b^+(\partial E_{\pm}))$, in particular $[f_+]=[id]\in\pi_0(\Isom_b^+(\partial E_+))\simeq\pi_0(SO(n))\simeq0$ and hence $[\varphi_0\circ f_+^{-1}\circ\varphi_0^{-1}]=[id]$. Thus we have $[\varphi_1\circ\varphi_0^{-1}]=[f_-]\in\pi_0(\Isom_b^+(\partial E_-))$, the contradiction.
   \end{proof}
\begin{rem}\label{Eells-Kuiper}
Let $p\in E_{\varphi}$ be the center of the disk $E_+=D^n$, $g_{\varphi}$ be a bundle-like metric for the singular Riemannian foliation $(E_{\varphi}, \fol_{\varphi})$. Then $(E_{\varphi}, g_{\varphi})$ is a Blaschke manifold at $p$. Conversely, a Blaschke manifold at $p$ is of the form $E_{\varphi}$ (cf. \cite{Be78}). Moreover, for $n=8$, the Eells-Kuiper quaternionic projective planes are of the form $E_{\varphi}$ for some certain $E_-$ (some $D^4$-bundles over $\sphere^4$ with boundary diffeomorphic to $\sphere^7$). On each of these manifolds, say $E_{\varphi_0}=E_+\cup_{\varphi_0}E_-$, Tang and Zhang \cite{TZ14} showed there exists a $\varphi_1$ such that $(E_{\varphi_1}, g_{\varphi_1})$ is an $SC^p$ manifold. Notice that the diffeomorphism types of the Eells-Kuiper quaternionic projective planes do not depend on the choice of the diffeomorphisms $\varphi:\partial E_+=\sphere^7\rightarrow \partial E_-$ (cf. \cite{KS07}). Therefore, $E_{\varphi}\cong E_{\varphi_0}\cong E_{\varphi_1}$ for any diffeomorphism $\varphi:\partial E_+=\sphere^7\rightarrow \partial E_-$. It is interesting to study how many non-equivalent foliations $(E_{\varphi}, \fol_{\varphi})$ on each of the Eells-Kuiper quaternionic projective planes and which of them induce (non-isometric) $SC^p$ Riemannian structures as $(E_{\varphi_1}, g_{\varphi_1})$.
\end{rem}

We do not know whether the construction in (2) of Corollary \ref{E+=Dn} yields non-equivalent foliations if we ignore orientation. Because in that case, we do not know, in general, whether the isotopy class $[\varphi_0\circ f_+^{-1}\circ\varphi_0^{-1}]$ for an orientation-reversing $[f_+]\in \pi_0(O(n))$ would lie in $\pi_0(\Isom_b(\partial E_-))$. However, when $E_{\pm}=D^n$, the question at hand can be resolved by Theorem \ref{negative ex}. Applying in addition Remark \ref{rem-diffSn} and (1) of Corollary \ref{E+=Dn}, we obtain
\begin{thm}\label{cor-twopoints isop}
\begin{itemize}
\item[(1)] Every homotopy sphere $\Sigma^n$ ($n\neq4,5$) and $\sphere^4$ admit exactly one equivalence class of isoparametric foliations with two points as the focal submanifolds.
\item[(2)] $\pi_0(\Diff(\sphere^{4}))\neq\mathbb{Z}_2$ if and only if $\sphere^5$ admits non-equivalent isoparametric foliations with two points as the focal submanifolds.
\end{itemize}
\end{thm}
\begin{proof}
 To prove (1), the only left stuff is to recall the well-known fact that every homotopy $n$-sphere except for exotic $4$-spheres (if exist) is a twisted $n$-sphere, the homotopy $n$-sphere obtained by gluing two $n$-disks through a diffeomorphism $\varphi\in \Diff^+(\sphere^{n-1})$. We recall that exotic $4$-spheres (if exist) admit no singular Riemannian foliations, of codimension 1 by \cite{GT12}, or of general codimension by \cite{GR13}.

 The necessity for (2) follows directly from (2) of Theorem \ref{negative ex}. On the other hand, (1) of Theorem \ref{negative ex} applies to prove the sufficiency, since the condition there is automatically satisfied by using the disk theorem of Palais as discussed right before Corollary \ref{E+=Dn}.
\end{proof}

The left-hand condition in (2) of Theorem \ref{cor-twopoints isop}, $\pi_0(\Diff(\sphere^{4}))\neq\mathbb{Z}_2$, means that $\Diff(\sphere^4)$ would have more than two components, and there would exist non-isotopic but pseudo-isotopic diffeomorphisms on $\sphere^{4}$.

In the rest part of this section, we present new examples of foliations on $\sphere^n$ that are non-equivalent to the classical isoparametric foliations on the unit sphere $\sphere^n(1)$, which illustrate the complexity for the step (1) in the classification problem as well as for the other two steps discussed above.

Recalling the criterion for distinguishing two foliations in Corollary \ref{non-isomorphic bundles}, we know if there exist two non-equivalent (up to bundle maps) disk bundles $E_+$ and $\widetilde{E}_+$ with diffeomorphic total spaces (thus let the total spaces be identified, say $N_+$), then for any other disk bundle $E_-$ with a diffeomorphism between boundaries $\varphi:\partial E_+=\partial \widetilde{E}_+=\partial N_+\rightarrow \partial E_-$, the canonical foliations of $E_+$ and $\widetilde{E}_+$ give rise to two non-equivalent foliations $\fol_{\varphi}$ and $\widetilde{\fol}_{\varphi}$ on $N:=N_+\cup_{\varphi}E_-$. In particular, this leads to the following new examples of foliations on $\sphere^n$.
\begin{exmp}\label{examples}
$(1)$  $\mathbb{S}^{n}$ $(n=11,12,13,15,16,17)$ admit isoparametric foliations (under some metrics) non-equivalent to the classical isoparametric foliations (under the round metric), with the same isoparametric hypersurfaces (diffeomorphic to $\sphere^{k}\times \sphere^{l}$) and the same focal submanifolds (diffeormophic to $(\sphere^{k}, \sphere^{l})$), for $(k,l)=(7,3)$, $(8,3)$, $(9,3)$, $(11,3)$, $(11,4)$, $(11,5)$, respectively.

$(2)$ $\sphere^{14}$ admits $15$ (ignore orientation) non-equivalent isoparametric foliations whose isoparametric hypersurfaces are diffeomorphic to $\sphere^7\times\sphere^6$ and focal submanifolds are $(\Sigma^7,\sphere^6)$, where $\Sigma^7\in\Theta_7\cong\mathbb{Z}_{28}$ is any homotopy $7$-sphere.
\end{exmp}
\begin{proof}
$(1)$ Haefliger and Levine's examples (cf. \cite{DW00}) ensure that there are non-trivial disk bundle structures on the total space of the trivial disk bundle $\sphere^{k}\times D^{l+1}$ for these $(k,l)$ listed. Letting $D^{k+1}\times \sphere^l$ be the other disk bundle $E_-$ and $\varphi$ be the identity map of their common boundary $\sphere^k\times\sphere^l$ yields the conclusion by the discussion above.

$(2)$ Notice that the tangent (disk) bundle of each homotopy sphere is diffeomorphic to that of the standard sphere (cf. \cite{DW00}), and clearly they are non-equivalent bundles if the homotopy spheres are non-diffeomorphic to each other. Recall that there are exactly $15$ non-diffeomorphic homotopy $7$-spheres. The tangent bundle $T\sphere^7$ is trivial and thus the total space of its disk bundle is $\sphere^7\times D^7$. Letting $D^8\times\sphere^6$ be the other disk bundle $E_-$ and $\varphi$ be the identity map of their common boundary $\sphere^7\times\sphere^6$ yields the conclusion by the discussion above.
\end{proof}

At last, we remind the reader that, as generalizations of classical isoparametric hypersurfaces in unit spheres, closed Dupin hypersurfaces also divide the sphere into two disk bundles (cf. \cite{Th83}, \cite{GH87}). It is still unknown whether these disk bundles are equivalent to those induced by classical isoparametric hypersurfaces.

 \section{Isoparametric foliations on homotopy spheres}\label{sect-isop-sphere}
 In order to prove Theorem \ref{main thm}, it suffices to consider homotopy $n$-spheres for $n\geq7$, since there are no exotic $n$-spheres for $n\leq6$ and $n\neq4$. In this case, the group $\Theta_n$ of h-cobordism classes of oriented homotopy $n$-spheres (always isomorphic to the group $\Gamma_n$ of oriented twisted $n$-spheres) is isomorphic to the mapping class group $\pi_0\Diff^+(\mathbb{S}^{n-1})$ by (cf. \cite{Ce70})
 \begin{eqnarray}\label{isom-spheres}
&\pi_0\Diff^+(\mathbb{S}^{n-1})&\longrightarrow \Gamma_n\simeq\Theta_n\\
&[\phi] &\longmapsto  \Sigma_{\phi}:=D^n\cup_{\phi}D^n.\nonumber
 \end{eqnarray}
Note that $\Sigma_{\phi}$ depends only on the isotopy class of $\phi\in\Diff^+(\mathbb{S}^{n-1})$.

\textbf{Proof of Theorem \ref{main thm}.} Given an isoparametric foliation $\fol$ on a homotopy $n$-sphere $\Sigma$, we have the decomposition (\ref{gluing str}) with $N=\Sigma=\Sigma_{+}\cup\Sigma_{-}\cong E_{\varphi}=E_{+}\cup_{\varphi}E_{-}$. In the following, for the sake of simplicity, we fix the orientations and (foliated) diffeomorphisms are assumed orientation-preserving.

It is well-known that any orientation-preserving diffeomorphism is isotopic to one that restricts to the identity on an embedded disk. Thus, for any $\phi\in\Diff^+(\mathbb{S}^{n-1})$, we can assume $\phi:\mathbb{S}^{n-1}=D_{+}^{n-1}\cup_{id}D_{-}^{n-1}\rightarrow D_{+}^{n-1}\cup_{id}D_{-}^{n-1}$ satisfy $\phi|_{D_{+}^{n-1}}=id$. Consider a disk $D_-^{n-1}$ in $M:=\partial E_{-}$ and write $M$ as $M=M'\cup D_{-}^{n-1}$. Define a diffeomorphism $d_{\phi}$ on $M$ by setting $d_{\phi}|_{D_{-}^{n-1}}=\phi|_{D_{-}^{n-1}}$ and identity on $M'$. Then gluing $E_{-}$ with $E_{+}$ by $d_{\phi}\circ\varphi$, we get a manifold
\[E_{d_{\phi}\circ\varphi}:=E_{+}\cup_{d_{\phi}\circ\varphi}E_{-}.\]
It follows easily that $E_{d_{\phi}\circ\varphi}$ depends only on the isotopy class of $\phi\in\Diff^+(\mathbb{S}^{n-1})$ and hence we have defined a map $\Phi_{\varphi}:\pi_0\Diff^+(\mathbb{S}^{n-1})\simeq\Theta_n\rightarrow\Phi_{\varphi}(\Theta_n)$ by mapping $[\phi]\simeq\Sigma_{\phi}$ to $E_{d_{\phi}\circ\varphi}$. Since the induced map $(d_{\phi})_{*}$ on homotopy groups $\pi_{*}(M)$ and homology groups $H_{*}(M)$ is trivial, it follows from van Kampen theorem and Mayer-Vietoris sequence that $E_{d_{\phi}\circ\varphi}$ is also a homotopy $n$-sphere. We claim that $E_{d_{\phi}\circ\varphi}$ is essentially diffeomorphic to the homotopy $n$-sphere $\Sigma_{\phi}\#\Sigma$. Consequently, the image $\Phi_{\varphi}(\Theta_n)=\Theta_n+\Sigma=\Theta_n$, which means that each homotopy $n$-sphere can be decomposed into the same disk bundles $E_{\pm}$. Moreover, it follows from the discussion in the last section that, the 1-1 correspondence $\{(\Sigma,\fol)\}\leftrightarrow \{(\widetilde{\Sigma},\widetilde{\fol})\}$ between the sets of equivalence classes of isoparametric foliations on any two homotopy $n$-spheres $\Sigma$ and $\widetilde{\Sigma}=\Sigma_{\phi}\#\Sigma$ can be represented by
\begin{equation*}
(E_{\varphi},\fol_{\varphi})\mapsto (E_{d_{\phi}\circ\varphi}, \fol_{d_{\phi}\circ\varphi}),
\end{equation*}
and the inverse is in the same form if we wrote $\Sigma=\Sigma_{\phi^{-1}}\#\widetilde{\Sigma}$. It is still left to show that the 1-1 correspondence is well-defined, i.e., if $(E_{\varphi_0},\fol_{\varphi_0})\cong (E_{\varphi_1},\fol_{\varphi_1})$ then $(E_{d_{\phi}\circ\varphi_0},\fol_{d_{\phi}\circ\varphi_0})\cong (E_{d_{\phi}\circ\varphi_1},\fol_{d_{\phi}\circ\varphi_1})$. By Proposition \ref{equivalence criterion} and Corollary \ref{isotopy invariant}, there are diffeomorphisms $h_{\pm}\in\Diff_c^+(\partial E_{\pm})$ such that $\varphi_1=h_{-}\circ\varphi_0\circ h_{+}^{-1}$ and without loss of generality, we can assume $h_-|_{D_-^{n-1}}=id$. Then $d_{\phi}\circ h_-=h_-\circ d_{\phi}$ on $M$ and hence $d_{\phi}\circ\varphi_1=h_{-}\circ d_{\phi}\circ\varphi_0\circ h_{+}^{-1}$, proving $(E_{d_{\phi}\circ\varphi_0},\fol_{d_{\phi}\circ\varphi_0})\cong (E_{d_{\phi}\circ\varphi_1},\fol_{d_{\phi}\circ\varphi_1})$.

Now we come back to consider the claim that $E_{d_{\phi}\circ\varphi}\cong \Sigma_{\phi}\#E_{\varphi}$. It requires essentially no more than an alternative explanation of the connected sum $N\#\Sigma_{\phi}$ by removing a disk in $N$ and gluing the disk back through $\phi$. In fact, we will prove it in a more general setting by (\ref{Nphi=N+phi}) in Section \ref{sect-inertia}, where replacing $N$ by $E_{\varphi}$ proves the claim here. For the sake of inspiration, we present here an h-cobordism proof for the case when $E_{\varphi}=\mathbb{S}^n$. That is,
we construct an explicit h-cobordism between $E_{d_{\phi}\circ\varphi}$ and $\Sigma_{\phi}\cong\Sigma_{\phi}\#\mathbb{S}^n$ when $E_{\varphi}=\mathbb{S}^n$.

Take a disk $D^n$ in $E_{-}$ without intersecting the boundary $\partial E_-$, and write $E_{-}$ as $E_{-}=E_{-}'\cup D^n$. Then $\partial E_{-}'=M\sqcup \mathbb{S}^{n-1}$. Take disks $D_{-}^{n-1}$ in $M$ and $\mathbb{S}^{n-1}$ respectively and connect them by an embedding of $D_{-}^{n-1}\times[0,1]$ in $E_{-}'$. Extend $\phi|_{D_{-}^{n-1}}$ trivially to $D_{-}^{n-1}\times[0,1]$ and by the identity elsewhere so as to define a diffeomorphism $\Psi_{\phi}$ on $E_{-}'$, thereby $\Psi_{\phi}|_M=d_{\phi}$ and $\Psi_{\phi}|_{\mathbb{S}^{n-1}}=\phi$.
Let $D(E_{-})=E_{-}\cup_{id} E_{-}$ be the double of $E_{-}$. Then it is an $\mathbb{S}^{m_{-}}$ bundle over $M_{-}$ bounding a disk bundle, say $W_{-}$, over $M_{-}$ of rank $m_{-}+1$ (see \cite{TXY12} for an interesting study of the topology and geometry on this double). Rewrite the boundaries $\partial W_{-}$ and $\partial D^{n+1}$ as
\begin{eqnarray*}
&\partial W_{-}=D(E_{-})=E_{-}\bigcup_ME_{-}'\bigcup_{\mathbb{S}^{n-1}} D^n,\\
&\partial D^{n+1}=\mathbb{S}^n=E_{\varphi}=E_{+}\bigcup_{\varphi}E_{-}'\bigcup_{\mathbb{S}^{n-1}} D^n.
\end{eqnarray*}
Gluing $W_{-}$ with $D^{n+1}$ along the common part $E_{-}'$ of their boundaries by the diffeomorphism $\Psi_{\phi}$ (and smoothing the corners), we get a manifold $W^{n+1}=D^{n+1}\cup_{\Psi_{\phi}}W_{-}$ with boundary $\partial W= E_{d_{\phi}\circ\varphi}\sqcup\Sigma_{\phi}$.

By van Kampen theorem $W$ is simply connected and thus $H_1(W)=0$. Since both $E_{-}$ and $W_{-}$ contract to $M_{-}$, we have $H_k(E_{-})\simeq H_k(W_{-})\simeq H_k(M_{-})$ for any $k$ and they vanish when $k\geq n-1$ since $dim(M_{-})<n-1$. By the Mayer-Vietoris sequence, we have $H_k(E_{-}')\simeq H_k(E_{-})\simeq H_k(W_{-})$ for $k=1,\cdots,n-2$, $H_{n-1}(E_{-}')\simeq\mathbb{Z}$ and $H_n(E_{-}')=0$. By the Mayer-Vietoris sequence
\[\cdots\rightarrow H_k(E_{-}')\rightarrow H_k(W_{-})\rightarrow H_k(W)\rightarrow H_{k-1}(E_{-}')\rightarrow H_{k-1}(W_{-})\rightarrow\cdots,\]
we obtain
\begin{equation*}
H_*(W)=\left\{\begin{array}{ll}
\mathbb{Z} & if~~*=0,n,\\
0 & otherwise.
\end{array}\right.
\end{equation*}
Applying the exact sequence of relative homology gives $H_*(W,\Sigma_{\phi})=0$. Therefore, by the h-cobordism theorem (cf. \cite{Mi65}), $E_{d_{\phi}\circ\varphi}$ is diffeomorphic to $\Sigma_{\phi}$.

The proof is now complete.\hfill $\Box$

 \section{Inertia groups and exotic smooth structures}\label{sect-inertia}
As in the last section, we always choose a representative $\phi$ of $[\phi] \in\pi_0(\Diff^+(\sphere^{n-1}))$ (identified with $\Sigma_{\phi}\in\Theta_n$ by (\ref{isom-spheres})) such that $\phi|_{D_+^{n-1}}=id$. Diffeomorphisms and embeddings with codimension zero are orientation-preserving.

Let $M^n$ be a closed oriented manifold. Recall (cf. \cite{Le70}) that there are two subgroups $I_0(M)\subset\Theta_n$, $I_1(M)\subset\Theta_{n+1}$ called the \emph{inertia groups} of $M$. $I_0(M)$ consists of all $\Sigma_{\phi}\in\Theta_n$ such that $M\#\Sigma_{\phi}\cong M$. $I_1(M)$ consists of all $\Sigma_{\phi}\in\Theta_{n+1}$ such that the diffeomorphism $d_{\phi}$ of $M$ which differs from the identity only on an $n$-disk in $M$ (identified with $D_-^n\subset\sphere^n$), and there coincides with $\phi$, is concordant to the identity. It is not hard to see that for $\Sigma_{\phi}\in\Theta_n\setminus I_0(M)$, $M\#\Sigma_{\phi}\ncong M$ is homeomorphic to $M$ (by a radial extension of $\phi$ to $D^n$) and hence gives an exotic oriented smooth structure on $M$. Moreover, different cosets in $\Theta_n/I_0(M)$ give distinct oriented smooth structures on $M$. Therefore, there exist at least $|\Theta_n|/|I_0(M)|$ distinct oriented smooth structures on $M$. $I_0(M)$ is then of great importance in the study of exotic smooth structures on $M$. On the other hand, $I_1(M)$ contributes to the study of the group $\Gamma(M)$ of concordance classes of diffeomorphisms of $M$, i.e., the coset space $\Theta_{n+1}/I_1(M)$ corresponds to a subset of $|\Theta_{n+1}|/|I_1(M)|$ elements in $\Gamma(M)$ (and hence in $\pi_0(\Diff^+(M))$).

Levine \cite{Le70} showed these two inertia groups have a very close relation:
\begin{equation}\label{I1=I0}
I_1(M)=I_0(M\times\sphere^1).
\end{equation}
In the following we relate further $I_1(M)$ with $I_0(N)$ when $M$ is a hypersurface in $N$, which leads to a proof of Theorem \ref{exoticstr-thm} and would certainly induce more applications.
\begin{thm}\label{thm-I1<I0}
Let $M^{n-1}$ be a closed oriented hypersurface embedded in a closed oriented manifold $N^n$. Then
\begin{equation}\label{I1<I0}
I_1(M^{n-1})\subseteq I_0(N^n).
\end{equation}
Therefore, $I_0(M^{n-1}\times\sphere^1)\subseteq I_0(N^n)$. In particular, there exist at least $|\Theta_n|/|I_0(N)|$ distinct oriented smooth structures on $M^{n-1}\times\sphere^1$, and there exist at least $|\Theta_n|/|I_0(N)|$ elements in $\Gamma(M)$ and $\pi_0(\Diff^+(M))$.
\end{thm}
\begin{proof}
Let $\widetilde{N}$ be the complementary of $M$ in $N$ with boundary $M\sqcup -M$. Given any $f\in\Diff^+(M)$, one gets a closed oriented manifold $N_f$ by gluing $\widetilde{N}$ along $M$ through $f$; thus $N_{id}=N$. For $\Sigma_{\phi}\in\Theta_n$, we denote by $N_{\phi}$ the manifold $N_{d_{\phi}}$, where $d_{\phi}\in\Diff^+(M)$ is the diffeomorphism that equals $\phi$ in a disk $D_-^{n-1}\subset M$ and $id$ outside it, as in the definition of $I_1(M)$. We claim that
\begin{equation}\label{Nphi=N+phi}
N_{\phi}\cong N\#\Sigma_{\phi}.
\end{equation}

Consider a collar $C:=M\times [0,\infty)$ of $M$ in $\widetilde{N}$ and the embedded disk $\widetilde{D}^n:= D_-^{n-1}\times [0,1]\cong D^n$ in $C$, where the diffeomorphism $\tilde{d}:D^n\rightarrow\widetilde{D}^n$ can be chosen so that it restricts to the identity on the common part $D_-^{n-1}$ of the boundaries $\partial \widetilde{D}^n$ and $\partial D^n=\mathbb{S}^{n-1}=D^{n-1}_+\cup D_-^{n-1}$. Otherwise, by the disk theorem, there is a diffeomorphism $f\in \Diff^+(\sphere^{n-1})$ isotopic to the identity such that $f|_{D_-^{n-1}}=\tilde{d}^{-1}|_{D_-^{n-1}}$. Let $F\in \Diff^+(D^n)$ be an extension of $f$. Then we can choose the diffeomorphism $\tilde{d}\circ F:D^n\rightarrow\widetilde{D}^n$ which restricts to the identity on $D_-^{n-1}$. Now since the diffeomorphism $\tilde{d}_{\phi}:=\tilde{d}\circ\phi\circ\tilde{d}^{-1}|_{\partial \widetilde{D}^n}\in\Diff^+(\partial\widetilde{D}^n)$ restricts to $\phi$ on $D_-^{n-1}$ and $id$ elsewhere, we can regard the manifold $N_{\phi}$ as
\[N_{\phi}=(N\setminus \widetilde{D}^n)\cup_{\tilde{d}_{\phi}}\widetilde{D}^n.\]
 On the other hand, by the disk theorem we can write the connected sum as
 \[N\#\Sigma_{\phi}=(N\setminus\widetilde{D}^n)\cup_{d_+\circ\tilde{d}^{-1}|_{\partial\widetilde{D}^n}}(\Sigma_{\phi}\setminus\overline{D}_+^n),\]
 where $d_{\pm}:D^n\rightarrow\overline{D}_{\pm}^n$ are disks embedded in $\Sigma_{\phi}=\overline{D}_+^n\cup\overline{D}_-^n$ with $d_-^{-1}\circ d_+|_{\partial D^n}=\phi$. Then the equation (\ref{Nphi=N+phi}) follows from
 \[(N\setminus\widetilde{D}^n)\cup_{d_+\circ\tilde{d}^{-1}|_{\partial\widetilde{D}^n}}\overline{D}_{-}^n\cong (N\setminus\widetilde{D}^n)\cup_{d_-^{-1}\circ d_+\circ\tilde{d}^{-1}|_{\partial\widetilde{D}^n}}D^n\cong (N\setminus \widetilde{D}^n)\cup_{\tilde{d}_{\phi}}\widetilde{D}^n.\]

 Now for any $\Sigma_{\phi}\in I_1(M)$, $N_{\phi}=N_{d_{\phi}}\cong N$ since $d_{\phi}$ is concordant to $id$. Then by (\ref{Nphi=N+phi}) we obtain $\Sigma_{\phi}\in I_0(N)$, proving (\ref{I1<I0}).

 The proof is now complete.
\end{proof}

\textbf{Proof of Theorem \ref{exoticstr-thm}.} Let $M$ be a closed hypersurface embedded in $\sphere^n$. Then $M$ is orientable and we fix its orientation (cf. \cite{MS74}). For any $k\geq0$ and any product $P^k:=\sphere^{k_1}\times\cdots\times\sphere^{k_l}$ of standard spheres of total dimension $k=\sum_{i=1}^l k_i$ ($k_i\geq1$) ($P^k$ is a point when $k=0$), we can embed $M^{n-1}\times P^k$ in $\sphere^n\times P^k$. Then by Theorem \ref{thm-I1<I0}, we have $I_0(M^{n-1}\times P^k\times\sphere^1)\subseteq I_0(\sphere^n\times P^k)$. By the theorem of Schultz \cite{Sc71} we obtain $I_0(M^{n-1}\times P^k\times\sphere^1)=I_0(\sphere^n\times P^k)=0$, completing the proof. \hfill $\Box$

 At last we remark that combining the relations (\ref{I1=I0}), (\ref{I1<I0}) with (\ref{Nphi=N+phi}) would deduce more applications than what we have shown, once provided with some known $I_0(N)$.

\begin{ack}
The author is very grateful to Professors Alexander Lytchak, Zizhou Tang, Gudlaugur Thorbergsson for their supports and valuable discussions.
\end{ack}


\end{document}